\documentclass[reqno]{amsart}%
\usepackage{amssymb,mathtools}
\usepackage{ifpdf}
\ifpdf
 \usepackage[hyperindex]{hyperref}
\else
 \expandafter\ifx\csname dvipdfm\endcsname\relax
 \usepackage[hypertex,hyperindex]{hyperref}%
 \else
 \usepackage[dvipdfm,hyperindex]{hyperref}%
 \fi
\fi

\allowdisplaybreaks[4]
\theoremstyle{plain}
\newtheorem{thm}{Theorem}[section]

\newtheorem{cor}{Corollary}[section]
\theoremstyle{remark}
\newtheorem{rem}{Remark}[section]

\numberwithin{equation}{section}

\DeclareMathOperator{\td}{d}
\DeclareMathOperator{\ti}{i}
\DeclareMathOperator{\te}{e}
\DeclareMathOperator{\sinc}{sinc}
\DeclareMathOperator{\sinhc}{sinhc}
\DeclareMathOperator{\sech}{sech}
\DeclareMathOperator{\csch}{csch}
\DeclareMathOperator{\arcsinh}{arcsinh}
\DeclareMathOperator{\arccosh}{arccosh}
\DeclareMathOperator{\arctanh}{arctanh}

\begin{document}

\title[Series expansions for real powers of sinc function]
{Several series expansions for real powers and several formulas for partial Bell polynomials of sinc and sinhc functions in terms of central factorial and Stirling numbers of second kind}

\author[F. Qi]{Feng Qi}
\address{Institute of Mathematics, Henan Polytechnic University, Jiaozuo 454010, Henan, China\newline\indent
School of Mathematical Sciences, Tiangong University, Tianjin 300387, China}
\email{\href{mailto: F. Qi <qifeng618@gmail.com>}{qifeng618@gmail.com}, \href{mailto: F. Qi <qifeng618@hotmail.com>}{qifeng618@hotmail.com}, \href{mailto: F. Qi <qifeng618@qq.com>}{qifeng618@qq.com}}
\urladdr{\url{https://qifeng618.wordpress.com}, \url{https://orcid.org/0000-0001-6239-2968}}

\author[P. Taylor]{Peter Taylor}
\address{Independent researcher, Valencia, Spain}
\email{\href{mailto: P. Taylor <pjt33@cantab.net>}{pjt33@cantab.net}}
\urladdr{\url{https://stackexchange.com/users/278703/peter-taylor}}

\begin{abstract}
In the paper, with the aid of the Fa\`a di Bruno formula, in terms of central factorial numbers of the second kind, and with the terminology of the Stirling numbers of the second kind, the authors derive several series expansions for any positive integer powers of the sinc and sinhc functions, discover several closed-form formulas for partial Bell polynomials of all derivatives of the sinc function, establish several series expansions for any real powers of the sinc and sinhc functions, and present several identities for central factorial numbers of the second kind and for the Stirling numbers of the second kind.
\end{abstract}

\keywords{Fa\`a di Bruno formula; series expansion; sinc function; sinhc function; positive integer power; real power; partial Bell polynomial; closed-form formula; Stirling number of the second kind; central factorial number of the second kind; weighted Stirling number of the second kind; combinatorial proof}

\subjclass{Primary 41A58; Secondary 05A19, 11B73, 11B83, 11C08, 33B10}

\thanks{This paper was typeset using\AmS-\LaTeX}

\maketitle
\tableofcontents

\section{Motivations}
According to common knowledge in complex analysis, the principal value of the number $\alpha^\beta$ for $\alpha,\beta\in\mathbb{C}$ with $\alpha\ne0$ is defined by $\alpha^\beta=\te^{\beta\ln\alpha}$, where $\ln\alpha=\ln|\alpha|+\ti\arg\alpha$ and $\arg\alpha$ are principal values of the logarithm and argument of $\alpha\ne0$ respectively. In what follows, we always consider principal values of real or complex functions discussed in this paper.
\par
In mathematical sciences, one commonly considers elementary functions
\begin{gather*}
\te^z, \quad \ln(1+z), \quad \sin z, \quad \csc z, \quad \cos z, \quad \sec z, \quad \tan z, \quad \cot z,\\
\arcsin z, \quad \arccos z, \quad \arctan z, \quad \sinh z, \quad \csch z, \quad \cosh z, \quad \sech z, \\
\tanh z, \quad \coth z,\quad \arcsinh z, \quad \arccosh z, \quad \arctanh z
\end{gather*}
and their series expansions at the point $z=0$. Their series expansions can be found in mathematical handbooks such as~\cite{abram, Gradshteyn-Ryzhik-Table-8th, NIST-HB-2010}.
\par
What are series expansions at $x=0$ of positive integer powers or real powers of these functions?
\par
It is combinatorial knowledge~\cite{Charalambides-book-2002, Comtet-Combinatorics-74} that coefficients of the series expansion of the power function $(\te^z-1)^k$ for $k\in\mathbb{N}=\{1,2,\dotsc\}$ are the Stirling numbers of the second kind, while coefficients of the series expansion of the power function $[\ln(1+z)]^k$ for $k\in\mathbb{N}$ are the Stirling numbers of the first kind. In other words, the power functions $(\te^z-1)^k$ and $[\ln(1+z)]^k$ for $k\in\mathbb{N}$ are generating functions of the Stirling numbers of the first and second kinds.
\par
In the paper~\cite{Carlitz-Fibonacci-1980-I}, among other things, Carlitz introduced the notion of weighted Stirling numbers of the second kind $R(n,k,r)$. Carlitz also proved in~\cite{Carlitz-Fibonacci-1980-I} that the numbers $R(n,k,r)$ can be generated by
\begin{equation}\label{S(n=k=x)-dfn}
\frac{(\te^z-1)^k}{k!}\te^{rz}=\sum_{n=k}^\infty R(n,k,r)\frac{z^n}{n!}
\end{equation}
and can be explicitly expressed by
\begin{equation}\label{S(n=k=x)-satisfy-eq}
R(n,k,r)=\frac1{k!}\sum_{j=0}^k(-1)^{k-j}\binom{k}{j}(r+j)^n
\end{equation}
for $r\in\mathbb{R}$ and $n\ge k\in\mathbb{N}_0=\{0,1,2,\dotsc\}$.
Specially, when $r=0$, the quantities $R(n,k,0)$ become the Stirling numbers of the second kind $S(n,k)$. By the way, the notion
$$
{n\brace k}_r=R(n-r,k-r,r)
$$
is called the $r$-Stirling numbers of the second kind in~\cite{Broder} by Broder.
\par
The central factorial numbers of the second kind $T(n,\ell)$ for $n\ge \ell\in\mathbb{N}_0$ can be generated~\cite{BSSV-NFAO-1989, Merca-Period-2016} by
\begin{equation}\label{T(n-k)-Gen-Eq}
\frac{1}{\ell!}\biggl(2\sinh\frac{z}{2}\biggr)^\ell=\sum_{n=\ell}^{\infty}T(n,\ell)\frac{z^n}{n!}.
\end{equation}
In~\cite[Chapter~6, Eq.~(26)]{Riordan-B-1968}, it was established that
\begin{equation}\label{T(n-k)-EF-Eq(3.1)}
T(n,\ell)=\frac1{\ell!} \sum_{j=0}^{\ell}(-1)^{j}\binom{\ell}{j}\biggl(\frac{\ell}2-j\biggr)^n.
\end{equation}
Note that $T(0,0)=1$ and $T(n,0)=0$ for $n\in\mathbb{N}$.
See also~\cite[Proposition~2.4, (xii)]{BSSV-NFAO-1989} and~\cite{msaen-566448.tex, centr-bell-polyn.tex}. Comparing~\eqref{T(n-k)-Gen-Eq} with~\eqref{S(n=k=x)-dfn} or comparing~\eqref{T(n-k)-EF-Eq(3.1)} with~\eqref{S(n=k=x)-satisfy-eq} gives the relation
\begin{equation}\label{Weight-Stirl-2nd-Cent-Fact-No}
R\biggl(n,\ell,-\frac{\ell}{2}\biggr)=T(n,\ell)
\end{equation}
between weighted Stirling numbers of the second kind and central factorial numbers of the second kind.
See also~\cite[Theorem~3.1]{centr-bell-polyn.tex}.
\par
In the handbook~\cite{Gradshteyn-Ryzhik-Table-8th}, series expansions at $z=0$ of the functions $\arcsin^2z$, $\arcsin^3z$, $\sin^2z$, $\cos^2z$, $\sin^3z$, and $\cos^3z$ are collected.
\par
In the papers~\cite{Borwein-Chamberland-IJMMS-2007, AADM-3164.tex, AIMS-Math20210491.tex, Migram-IFTS-2006, Taylor-arccos-v2.tex, Wilf-Ward-2010P.tex} and plenty of references collected therein, the series expansions at $z=0$ of the functions
$\arcsin^mz$, $\arcsinh^mz$, $\arctan^mz$, $\arctanh^mz$ for $m\in\mathbb{N}$ have been established, applied, reviewed, and surveyed.
\par
In the papers~\cite{Brychkov-ITSF-2009, Tan-Der-App-Thanks.tex}, explicit series expansions at $z=0$ of the functions $\tan^2z$, $\tan^3z$, $\cot^2z$, $\cot^3z$, $\sin^mz$, $\cos^mz$ for $m\in\mathbb{N}$ were written down.
\par
In the papers~\cite{Baricz-AML-2010, Bender-Brody-Meister-JMP-2003, Bess-Pow-Polyn-CMES.tex, Howard-Fibonacci-1985, Moll-Vignat-IJNT-2014, Thir-Nanj-1951-India, Yang-Zhen-MIA-2018-Bessel}, series expansions of the functions $I_\mu(z)I_\nu(z)$ and $[I_\nu(z)]^2$ were explicitly written out, while the series expansion of the power function $[I_\nu(z)]^r$ for $\nu\in\mathbb{C}\setminus\{-1,-2,\dotsc\}$ and $r,z\in\mathbb{C}$ was recursively formulated, where $I_\nu(z)$ denotes modified Bessel functions of the first kind.
\par
In the paper~\cite{Maclaurin-series-arccos-v3.tex}, series expansions at $z=0$ of the functions $(\arccos z)^r$ and $\bigl(\frac{\arcsin z}{z}\bigr)^r$ were established for real $r\in\mathbb{R}$. In~\cite{Taylor-arccos-v2.tex}, a series expansion at $z=1$ of the function $\bigl[\frac{(\arccos z)^{2}}{2(1-z)}\bigr]^r$ was invented for real $r\in\mathbb{R}$.
\par
For $z\in\mathbb{C}$, the functions
\begin{equation*}
\sinc z=
\begin{dcases}
\frac{\sin z}{z}, & z\ne0\\
1, & z=0
\end{dcases}
\end{equation*}
and
\begin{equation*}
\sinhc z=
\begin{dcases}
\frac{\sinh z}{z}, & z\ne0\\
1, & z=0
\end{dcases}
\end{equation*}
are called the sinc function and hyperbolic sinc function respectively. The function $\sinc z$ is also called the sine cardinal or sampling function, as well as the function $\sinhc z$ is also called hyperbolic sine cardinal, see~\cite{Sanchez-Reyes-CMJ-2012}.  The sinc function $\sinc z$ arises frequently in signal processing, the theory of the Fourier transforms, and other areas in mathematics, physics, and engineering. It is easy to see that these two functions $\sinc z$ and $\sinhc z$ are analytic on $\mathbb{C}$, that is, they are entire functions.
\par
In~\cite[Theorem~11.4]{Charalambides-book-2002} and~\cite[p.~139, Theorem~C]{Comtet-Combinatorics-74}, the Fa\`a di Bruno formula is given for $n\in\mathbb{N}$ and $z\in\mathbb{C}$ by
\begin{equation}\label{Bruno-Bell-Polynomial}
\frac{\td^n}{\td z^n}f\circ h(z)=\sum_{k=1}^nf^{(k)}(h(z)) B_{n,k}\bigl(h'(z),h''(z),\dotsc,h^{(n-k+1)}(z)\bigr),
\end{equation}
where partial Bell polynomials $B_{n,k}$ are defined for $n\ge k\in\mathbb{N}_0$ by
\begin{equation*}
B_{n,k}(z_1,z_2,\dotsc,z_{n-k+1})=\sum_{\substack{1\le i\le n-k+1\\ \ell_i\in\{0\}\cup\mathbb{N}\\ \sum_{i=1}^{n-k+1}i\ell_i=n\\
\sum_{i=1}^{n-k+1}\ell_i=k}}\frac{n!}{\prod_{i=1}^{n-k+1}\ell_i!} \prod_{i=1}^{n-k+1}\biggl(\frac{z_i}{i!}\biggr)^{\ell_i}
\end{equation*}
in~\cite[Definition~11.2]{Charalambides-book-2002} and~\cite[p.~134, Theorem~A]{Comtet-Combinatorics-74}.
\par
In this paper, with the help of the Fa\`a di Bruno formula~\eqref{Bruno-Bell-Polynomial}, in terms of central factorial numbers of the second kind $T(n,k)$, and with the terminology of the Stirling numbers of the second kind $S(n,k)$, we will derive several series expansions at $z=0$ of the positive integer power functions $\sinc^{\ell}z$ and $\sinhc^{\ell}z$ for $\ell\in\mathbb{N}$ and $z\in\mathbb{C}$, we will deduce several closed-form formulas for central factorial numbers of the second kind $T(j+\ell,\ell)$ with $j,\ell\in\mathbb{N}$ in terms of the Stirling numbers of the second kind $S(n,k)$, we will discover several closed-form formulas of specific partial Bell polynomials
\begin{equation*}
B_{n,k}\biggl(0,-\frac{1}{3},0,\frac{1}{5},\dotsc, \frac{(-1)^{n-k}}{n-k+2} \sin\frac{(n-k)\pi}{2}\biggr)
\end{equation*}
for $n\ge k\in\mathbb{N}$, we will establish series expansions at $z=0$ of the real power functions $\sinc^rz$ and $\sinhc^rz$ for $z\in\mathbb{C}$ and $r\in\mathbb{R}$, and we will present several identities for central factorial numbers of the second kind $T(n,k)$ and for the Stirling numbers of the second kind $S(n,k)$.

\section{Several series expansions of positive integer powers}

In this section, we derive several series expansions at $z=0$ of the positive integer power functions $\sinc^\ell z$ and $\sinhc^\ell z$ for $\ell\in\mathbb{N}$ and $z\in\mathbb{C}$ in terms of central factorial numbers of the second kind $T(n,k)$ and the Stirling numbers of the second kind $S(n,k)$, we deduce several closed-form formulas of $T(j+\ell,\ell)$ for $j\in\mathbb{N}_0$ and $\ell\in\mathbb{N}$ in terms of the Stirling numbers of the second kind $S(n,k)$, and we present several identities for central factorial numbers of the second kind $T(n,k)$.

\begin{thm}\label{Since-series-expan-lem}
For $\ell\in\mathbb{N}_0$ and $z\in\mathbb{C}$, we have
\begin{equation}\label{sine-power-ser-expan-eq}
\sinc^{\ell}z
=1+\sum_{j=1}^{\infty} (-1)^{j}\frac{T(\ell+2j,\ell)}{\binom{\ell+2j}{\ell}} \frac{(2z)^{2j}}{(2j)!}.
\end{equation}
\end{thm}

\begin{proof}
For $\ell\in\mathbb{N}$, the formula
\begin{equation}\label{sin-poer-exp}
\sin^\ell z=\frac{(-1)^\ell}{2^\ell}\sum_{q=0}^\ell(-1)^q\binom{\ell}{q} \cos\biggl[(2q-\ell)z-\frac{\ell\pi}2\biggr]
\end{equation}
is given in~\cite[Corollary~2.1]{Wallis-Ratio-Sum.tex}.
Applying the identity
\begin{equation*}
\cos(z-y)=\cos z\cos y+\sin z\sin y
\end{equation*}
to the formula~\eqref{sin-poer-exp} leads to
\begin{align*}
\sin^\ell z&=\frac{(-1)^\ell}{2^\ell}\sum_{q=0}^\ell(-1)^q\binom{\ell}{q} \biggl(\cos[(2q-\ell)z]\cos\frac{\ell\pi}2 +\sin[(2q-\ell)z]\sin\frac{\ell\pi}2\biggr)\\
&=\frac{(-1)^\ell}{2^\ell}\cos\frac{\ell\pi}2\sum_{q=0}^\ell(-1)^q\binom{\ell}{q} \Biggl[1+\sum_{j=1}^{\infty}(-1)^j(2q-\ell)^{2j}\frac{z^{2j}}{(2j)!}\Biggr] \\
&\quad+\frac{(-1)^\ell}{2^\ell} \sin\frac{\ell\pi}2 \sum_{q=0}^\ell(-1)^q\binom{\ell}{q} \sum_{j=0}^{\infty}(-1)^j(2q-\ell)^{2j+1}\frac{z^{2j+1}}{(2j+1)!}\\
&=\frac{(-1)^\ell}{2^\ell}\cos\frac{\ell\pi}2\sum_{j=1}^{\infty}(-1)^j \Biggl[\sum_{q=0}^\ell(-1)^q\binom{\ell}{q} (2q-\ell)^{2j}\Biggr]\frac{z^{2j}}{(2j)!} \\
&\quad+\frac{(-1)^\ell}{2^\ell} \sin\frac{\ell\pi}2 \sum_{j=0}^{\infty}(-1)^j \Biggl[\sum_{q=0}^\ell(-1)^q\binom{\ell}{q} (2q-\ell)^{2j+1}\Biggr]\frac{z^{2j+1}}{(2j+1)!}.
\end{align*}
Replacing $\ell$ by $2\ell-1$ and by $2\ell$ and simplifying result in the series expansions
\begin{equation}\label{sin-poer-expansion-odd}
\sin^{2\ell-1}z=\frac{(-1)^\ell}{2^{2\ell-1}} \sum_{j=\ell}^{\infty}(-1)^{j-1} \Biggl[\sum_{q=0}^{2\ell-1}(-1)^q\binom{{2\ell-1}}{q} (2q-2\ell+1)^{2j-1}\Biggr]\frac{z^{2j-1}}{(2j-1)!}
\end{equation}
and
\begin{equation}\label{sin-poer-expansion-even}
\sin^{2\ell} z=\frac{(-1)^\ell}{2^{2\ell}}\sum_{j=\ell}^{\infty}(-1)^j2^{2j} \Biggl[\sum_{q=0}^{2\ell}(-1)^q\binom{{2\ell}}{q} (q-\ell)^{2j}\Biggr]\frac{z^{2j}}{(2j)!}.
\end{equation}
The series expansions~\eqref{sin-poer-expansion-odd} and~\eqref{sin-poer-expansion-even} can be reformulated as
\begin{equation*}
\sinc^{2\ell-1}z
=\frac{1}{2^{2\ell-1}} \sum_{j=0}^{\infty}\frac{(-1)^{j-1}}{(2\ell+2j-1)!} \Biggl[\sum_{q=0}^{2\ell-1}(-1)^q\binom{{2\ell-1}}{q} (2q-2\ell+1)^{2\ell+2j-1}\Biggr]z^{2j}
\end{equation*}
and
\begin{equation*}
\sinc^{2\ell}z
=\frac{1}{2^{2\ell}}\sum_{j=0}^{\infty}\frac{(-1)^{j}}{(2\ell+2j)!} \Biggl[\sum_{q=0}^{2\ell}(-1)^q\binom{{2\ell}}{q} (2q-2\ell)^{2\ell+2j}\Biggr]z^{2j}.
\end{equation*}
for $\ell\in\mathbb{N}$ and $z\in\mathbb{C}$. These two series expansions can be unified and rearranged as the series expansion~\eqref{sine-power-ser-expan-eq}. Theorem~\ref{Since-series-expan-lem} is thus proved.
\end{proof}

\begin{cor}
For $\ell\in\mathbb{N}$, we have
\begin{equation*}
T(2j-1,2\ell-1)
=
\begin{dcases}
0, & 1\le j\le\ell-1\\
1, & j=\ell
\end{dcases}
\end{equation*}
and
\begin{equation*}
T(2j,2\ell)=
\begin{dcases}
0, & 1\le j\le\ell-1\\
1, & j=\ell
\end{dcases}
\end{equation*}
\end{cor}

\begin{proof}
This follows from Theorem~\ref{Since-series-expan-lem} and its proof.
\end{proof}

\begin{thm}
For $j,\ell\in\mathbb{N}_0$, we have
\begin{equation}\label{QGWSID3O}
T(2j+\ell+1,\ell)=0.
\end{equation}
For $\ell\in\mathbb{N}_0$ and $z\in\mathbb{C}$, the series expansions
\begin{equation}\label{sinh-power-ser}
\sinhc^\ell z=1+\sum_{j=1}^{\infty}\frac{T(2j+\ell,\ell)}{\binom{2j+\ell}{\ell}}\frac{(2z)^{2j}}{(2j)!}
\end{equation}
and~\eqref{sine-power-ser-expan-eq} are valid.
\end{thm}

\begin{proof}
Replacing $z$ by $2z$ in~\eqref{T(n-k)-Gen-Eq} and rearranging yield
\begin{equation*}
\biggl(\frac{\sinh z}{z}\biggr)^\ell=1+\sum_{n=1}^{\infty}\frac{T(n+\ell,\ell)}{\binom{n+\ell}{\ell}}\frac{(2z)^{n}}{n!}.
\end{equation*}
Considering that the function $\frac{\sinh x}{x}$ is even on $\mathbb{R}$, we conclude that the identity~\eqref{QGWSID3O} and the series~\eqref{sinh-power-ser} are valid.
\par
Substituting $z\operatorname{i}$ for $z$ in~\eqref{sinh-power-ser} and employing the relation $\sinh(z\operatorname{i})=\operatorname{i}\sin z$ give the series expansion~\eqref{sine-power-ser-expan-eq} in Theorem~\ref{Since-series-expan-lem}.
\end{proof}

\begin{thm}\label{sinc-power-series-2nd-thm}
For $j,\ell\in\mathbb{N}_0$ and $z\in\mathbb{C}$, we have
\begin{equation}\label{T-S(n-k)=0}
\sum_{k=0}^{2j+1}(-1)^k\binom{2j+1}{k}\biggl(\frac{2}{\ell}\biggr)^k \frac{S(k+\ell,\ell)} {\binom{k+\ell}{\ell}}=0
\end{equation}
and
\begin{equation}\label{sinc-power-series-2nd-Eq}
\sinc^\ell z=1+\sum_{j=1}^{\infty}(-1)^{j} \Biggl[\sum_{k=0}^{2j}(-1)^k\binom{2j}{k}\biggl(\frac{2}{\ell}\biggr)^k \frac{S(k+\ell,\ell)} {\binom{k+\ell}{\ell}}\Biggr] \frac{(\ell z)^{2j}}{(2j)!}.
\end{equation}
\end{thm}

\begin{proof}
Taking $r=0$ in~\eqref{S(n=k=x)-dfn} and reformulating give
\begin{equation}\label{S(nk)-dfn}
\biggl(\frac{\te^z-1}{z}\biggr)^k=\sum_{n=0}^\infty \frac{S(n+k,k)}{\binom{n+k}{k}}\frac{z^{n}}{n!}.
\end{equation}
Since
\begin{equation*}
\sin z=\frac{\te^{z\ti}-\te^{-z\ti}}{2\ti}=\frac{\te^{2z\ti}-1}{2\ti}\te^{-z\ti},
\end{equation*}
by~\eqref{S(nk)-dfn} and the Cauchy product of two series, we obtain
\begin{align*}
\sinc^\ell z&=\biggl(\frac{\sin z}{z}\biggr)^\ell
=\biggl(\frac{\te^{2z\ti}-1}{2z\ti}\biggr)^\ell\te^{-\ell z\ti}\\
&=\Biggl[\sum_{n=0}^\infty \frac{S(n+\ell,\ell)}{\binom{n+\ell}{\ell}}\frac{(2z\ti)^{n}}{n!}\Biggr] \Biggl[\sum_{n=0}^{\infty}\frac{(-\ell z\ti)^n}{n!}\Biggr]\\
&=\sum_{j=0}^{\infty}\Biggl[\sum_{n=0}^{j}\frac{S(n+\ell,\ell)}{\binom{n+\ell}{\ell}}\frac{(2\ti)^{n}}{n!} \frac{(-\ell \ti)^{j-n}}{(j-n)!}\Biggr]z^j\\
&=\sum_{j=0}^{\infty}(-1)^j\frac{\ell^j}{j!} \Biggl[\sum_{k=0}^{j}(-1)^k\binom{j}{k}\biggl(\frac{2}{\ell}\biggr)^k \frac{S(k+\ell,\ell)} {\binom{k+\ell}{\ell}}\Biggr]\biggl(\cos\frac{j\pi}2+\ti\sin\frac{j\pi}2\biggr)z^j\\
&=\sum_{j=0}^{\infty}(-1)^{j}\frac{\ell^{2j}}{(2j)!} \Biggl[\sum_{k=0}^{2j}(-1)^k\binom{2j}{k}\biggl(\frac{2}{\ell}\biggr)^k \frac{S(k+\ell,\ell)} {\binom{k+\ell}{\ell}}\Biggr]z^{2j}\\
&\quad+\ti\sum_{j=1}^{\infty}(-1)^{j}\frac{\ell^{2j-1}}{(2j-1)!} \Biggl[\sum_{k=0}^{2j-1}(-1)^k\binom{2j-1}{k}\biggl(\frac{2}{\ell}\biggr)^k \frac{S(k+\ell,\ell)} {\binom{k+\ell}{\ell}}\Biggr]z^{2j-1}\\
&=\sum_{j=0}^{\infty}(-1)^{j}\frac{\ell^{2j}}{(2j)!} \Biggl[\sum_{k=0}^{2j}(-1)^k\binom{2j}{k}\biggl(\frac{2}{\ell}\biggr)^k \frac{S(k+\ell,\ell)} {\binom{k+\ell}{\ell}}\Biggr]z^{2j}.
\end{align*}
The proof of Theorem~\ref{sinc-power-series-2nd-thm} is complete.
\end{proof}

\begin{cor}
For $j\in\mathbb{N}_0$ and $\ell\in\mathbb{N}$, we have
\begin{equation}\label{T(j+ell-ell)}
\frac{T(2j+\ell,\ell)}{\binom{2j+\ell}{\ell}}
=\sum_{m=0}^{2j}(-1)^{m}\binom{2j}{m} \biggl(\frac{\ell}{2}\biggr)^{m} \frac{S(2j+\ell-m,\ell)} {\binom{2j+\ell-m}{\ell}}.
\end{equation}
\end{cor}

\begin{proof}
This follows from comparing the series expansion~\eqref{sine-power-ser-expan-eq} in Theorem~\ref{Since-series-expan-lem} with the series expansion~\eqref{sinc-power-series-2nd-Eq} in Theorem~\ref{sinc-power-series-2nd-thm} and simplifying.
\end{proof}

\begin{cor}
For $j\in\mathbb{N}_0$ and $\ell\in\mathbb{N}$, we have
\begin{equation}\label{T(j+ell-ell)-j}
\frac{T(j+\ell,\ell)}{\binom{j+\ell}{\ell}}
=\sum_{m=0}^{j}(-1)^{m}\binom{j}{m} \biggl(\frac{\ell}{2}\biggr)^{m} \frac{S(j+\ell-m,\ell)} {\binom{j+\ell-m}{\ell}}.
\end{equation}
\end{cor}

\begin{proof}
This follows from combining the identities~\eqref{QGWSID3O}, \eqref{T-S(n-k)=0}, and~\eqref{T(j+ell-ell)}.
\end{proof}

\section{Closed-form formulas for specific partial Bell polynomials}

In this section, with the help of Theorem~\ref{Since-series-expan-lem} and other results in the above section, we establish several closed-form formulas for specific partial Bell polynomials $B_{n,k}$ of all derivatives at $z=0$ of the sinc function $\sinc z$.

\begin{thm}\label{part-Bell-sine-thm}
For $n\ge k\ge1$ and $m\in\mathbb{N}$, partial Bell polynomials $B_{n,k}$ satisfy
\begin{equation*}
B_{2m-1,k}\biggl(0,-\frac{1}{3},0,\frac{1}{5},\dotsc, \frac{(-1)^{m}}{2m-k+1}\cos\frac{k\pi}{2}\biggr)=0
\end{equation*}
and
\begin{multline*}
B_{2m,k}\biggl(0,-\frac{1}{3},0,\frac{1}{5},\dotsc, \frac{(-1)^m}{2m-k+2}\sin\frac{k\pi}{2}\biggr)\\*
=(-1)^{m+k}\frac{2^{2m}}{k!}\sum_{j=1}^k(-1)^j\binom{k}{j} \frac{T(2m+j,j)}{\binom{2m+j}{j}}.
\end{multline*}
\end{thm}

\begin{proof}
From
\begin{equation*}
\sinc z=\sum_{j=0}^{\infty}\frac{(-1)^j}{2j+1}\frac{z^{2j}}{(2j)!}, \quad z\in\mathbb{C},
\end{equation*}
it follows that
\begin{equation}\label{sine-deriv-0-Eq}
(\sinc z)^{(2j)}\big|_{z=0}=\frac{(-1)^j}{2j+1}
\quad\text{and}\quad
(\sinc z)^{(2j-1)}\big|_{z=0}=0
\end{equation}
for $j\in\mathbb{N}$.
\par
On~\cite[p.~133]{Comtet-Combinatorics-74}, the identity
\begin{equation}\label{113-final-formula}
\frac1{m!}\Biggl(\sum_{\ell=1}^\infty z_\ell\frac{t^\ell}{\ell!}\Biggr)^m =\sum_{n=m}^\infty B_{n,m}(z_1,z_2,\dotsc,z_{n-m+1})\frac{t^n}{n!}
\end{equation}
is given for $m\in\mathbb{N}_0$. The formula~\eqref{113-final-formula} implies that
\begin{equation}\label{Bell-Polyn-final-formula}
B_{n+k,k}(z_1,z_2,\dotsc,z_{n+1})
=\binom{n+k}{k}\lim_{t\to0}\frac{\td^{n}}{\td t^{n}}\Biggl[\sum_{\ell=0}^\infty \frac{z_{\ell+1}}{(\ell+1)!}t^{\ell}\Biggr]^k
\end{equation}
for $n\ge k\in\mathbb{N}_0$. Substituting $z_{2j}=\frac{(-1)^j}{2j+1}$ and $z_{2j-1}=0$, that is, $z_j=\frac{1}{j+1}\cos\bigl(\frac{j}{2}\pi\bigr)$, for $j\in\mathbb{N}$ into~\eqref{Bell-Polyn-final-formula} results in
\begin{align*}
&\quad B_{n+k,k}\biggl(0,-\frac{1}{3},0,\frac{1}{5},\dotsc, \frac{1}{n+2}\cos\biggl(\frac{n+1}{2}\pi\biggr)\biggr)\\
&=\binom{n+k}{k}\lim_{t\to0}\frac{\td^{n}}{\td t^{n}}\Biggl[\sum_{\ell=0}^\infty \frac{1}{(\ell+2)!}\cos\biggl(\frac{\ell+1}{2}\pi\biggr)t^{\ell}\Biggr]^k\\
&=\binom{n+k}{k}\lim_{t\to0}\frac{\td^{n}}{\td t^{n}}\biggl(\frac{\sinc t-1}{t}\biggr)^k\\
&=\binom{n+k}{k}\lim_{t\to0}\frac{\td^{n}}{\td t^{n}}\Biggl[\frac{(-1)^k}{t^k}+\frac{(-1)^k}{t^k}\sum_{j=1}^k(-1)^j\binom{k}{j}(\sinc t)^j\Biggr]\\
&=\binom{n+k}{k}\lim_{t\to0}\frac{\td^{n}}{\td t^{n}} \Biggl(\frac{(-1)^k}{t^k} \sum_{\ell=1}^{\infty}(-1)^\ell\Biggl[\sum_{j=1}^k\binom{k}{j} \frac{1}{2^{j}} \frac{1}{(j+2\ell)!}\\
&\quad\times\sum_{q=0}^{j}(-1)^q\binom{{j}}{q} (2q-j)^{j+2\ell}\Biggr]t^{2\ell}\Biggr)\\
&=(-1)^k\binom{n+k}{k}\lim_{t\to0}\frac{\td^{n}}{\td t^{n}} \sum_{\ell=k}^{\infty}(-1)^\ell \Biggl[\sum_{j=1}^k\binom{k}{j} \frac{1}{2^{j}} \frac{1}{(j+2\ell)!}\\
&\quad\times\sum_{q=0}^{j}(-1)^q\binom{{j}}{q} (2q-j)^{j+2\ell}\Biggr]t^{2\ell-k}\\
&=(-1)^k\binom{n+k}{k}\lim_{t\to0}\sum_{\ell=k}^{\infty}(-1)^\ell\Biggl[\sum_{j=1}^k\binom{k}{j} \frac{1}{2^{j}} \frac{1}{(j+2\ell)!}\\
&\quad\times\sum_{q=0}^{j}(-1)^q\binom{{j}}{q} (2q-j)^{j+2\ell}\Biggr]\langle2\ell-k\rangle_nt^{2\ell-k-n}\\
&=\begin{dcases}
0, \quad n+k=2m+1\\
(-1)^{k+m}\frac{(2m)!}{k!}\sum_{j=1}^k\binom{k}{j} \frac{1}{2^{j}(j+2m)!}\sum_{q=0}^{j}(-1)^q\binom{{j}}{q} (2q-j)^{j+2m}, \quad n+k=2m
\end{dcases}
\end{align*}
for $m\in\mathbb{N}$ and $n\ge k\ge1$, where we used the series expansion~\eqref{sine-power-ser-expan-eq} in Theorem~\ref{Since-series-expan-lem}. The proof of Theorem~\ref{part-Bell-sine-thm} is complete.
\end{proof}

\begin{cor}
For $k\ge2$ and $1\le\ell\le k-1$, we have
\begin{equation}\label{QGWSID2}
\sum_{j=1}^k(-1)^j\binom{k}{j}\frac{T(2\ell+j,j)}{\binom{2\ell+j}{j}}=0
\end{equation}
and
\begin{equation*}
\sum_{j=1}^k(-1)^j\binom{k}{j}\sum_{m=0}^{2\ell}(-1)^{m}\binom{2\ell}{m} \biggl(\frac{j}{2}\biggr)^{m} \frac{S(2\ell+j-m,j)} {\binom{2\ell+j-m}{j}}=0.
\end{equation*}
\end{cor}

\begin{proof}
This follows from the proof of Theorem~\ref{part-Bell-sine-thm} and further making use of the formula~\eqref{T(j+ell-ell)}.
\end{proof}

\begin{cor}\label{part-Bell-sine-cor}
For $n\ge k\ge1$, partial Bell polynomials $B_{n,k}$ satisfy
\begin{multline*}
B_{n,k}\biggl(0,-\frac{1}{3},0,\frac{1}{5},\dotsc, \frac{1}{n-k+2}\cos\biggl(\frac{n-k+1}{2}\pi\biggr)\biggr)\\*
=(-1)^{k}\cos\biggl(\frac{n\pi}{2}\biggr)\frac{2^{n}}{k!}\sum_{j=1}^k(-1)^j\binom{k}{j} \frac{T(n+j,j)}{\binom{n+j}{j}}
\end{multline*}
and
\begin{multline*}
B_{n,k}\biggl(0,-\frac{1}{3},0,\frac{1}{5},\dotsc, \frac{1}{n-k+2}\cos\biggl(\frac{n-k+1}{2}\pi\biggr)\biggr)\\*
=(-1)^{k}\cos\biggl(\frac{n\pi}{2}\biggr)\frac{2^{n}}{k!}\sum_{j=1}^k(-1)^j\binom{k}{j} \sum_{m=0}^{n}(-1)^{m}\binom{n}{m} \biggl(\frac{j}{2}\biggr)^{m} \frac{S(n+j-m,j)} {\binom{n+j-m}{j}}.
\end{multline*}
\end{cor}

\begin{proof}
This follows from combining the identity~\eqref{QGWSID3O} with Theorem~\ref{part-Bell-sine-thm} and the formula~\eqref{T(j+ell-ell)-j}.
\end{proof}

Applying Theorem~\ref{part-Bell-sine-thm} and Corollary~\ref{part-Bell-sine-cor}, we deduce the following corollary.

\begin{cor}\label{exp-sinc-derv-cor}
For $z\in\mathbb{C}$, we have
\begin{align*}
\te^{\sinc z-1}
&=1+\sum_{k=1}^{\infty}(-1)^k\Biggl[\sum_{j=1}^{2k} \frac{(-1)^{j}}{j!}\sum_{\ell=1}^j(-1)^\ell\binom{j}{\ell} \frac{T(2k+\ell,\ell)}{\binom{2k+\ell}{\ell}}\Biggr]\frac{(2z)^{2k}}{(2k)!}\\
&=1-\frac{z^2}{6}+\frac{z^4}{45}-\frac{107 z^6}{45360}+\frac{1189 z^8}{5443200}-\frac{1633 z^{10}}{89812800}+\dotsm
\end{align*}
and
\begin{multline*}
\te^{\sinc z-1}=1+\sum_{k=1}^{\infty}(-1)^k\Biggl[\sum_{j=1}^{2k} \frac{(-1)^{j}}{j!}\sum_{\ell=1}^j(-1)^\ell\binom{j}{\ell}\\
\times\sum_{m=0}^{2k}(-1)^{m}\binom{2k}{m} \biggl(\frac{\ell}{2}\biggr)^{m} \frac{S(2k+\ell-m,\ell)} {\binom{2k+\ell-m}{\ell}}\Biggr]\frac{(2z)^{2k}}{(2k)!}.
\end{multline*}
\end{cor}

\begin{proof}
Making use of the Fa\`a di Bruno formula~\eqref{Bruno-Bell-Polynomial}, the derivatives in~\eqref{sine-deriv-0-Eq}, and Theorem~\ref{part-Bell-sine-thm}, we obtain
\begin{align*}
\te^{\sinc z}&=\sum_{k=0}^{\infty}\biggl(\lim_{z\to0}\frac{\td^k\te^{\sinc z}}{\td z^k}\biggr)\frac{z^k}{k!}\\
&=\te+\sum_{k=1}^{\infty}\Biggl[\lim_{z\to0}\sum_{j=1}^{k}\te^{\sinc z} B_{k,j}\bigl((\sinc z)',(\sinc z)'',\dotsc, (\sinc z)^{(k-j+1)}\bigr)\Biggr]\frac{z^k}{k!}\\
&=\te+\te\sum_{k=1}^{\infty}\Biggl[\sum_{j=1}^{k} B_{k,j}\bigl((\sinc z)'\big|_{z=0},(\sinc z)''\big|_{z=0}, \dotsc, (\sinc z)^{(k-j+1)}\big|_{z=0}\bigr)\Biggr]\frac{z^k}{k!}\\
&=\te+\te\sum_{k=1}^{\infty}\Biggl[\sum_{j=1}^{k} B_{k,j}\biggl(0,-\frac{1}{3}, \dotsc, \frac{1}{k-j+2}\cos\biggl(\frac{k-j+1}{2}\pi\biggr)\biggr)\Biggr]\frac{z^k}{k!}\\
&=\te+\te\sum_{k=1}^{\infty}\Biggl[\sum_{j=1}^{2k} B_{2k,j}\biggl(0,-\frac{1}{3}, \dotsc, \frac{1}{2k-j+2}\cos\biggl(\frac{2k-j+1}{2}\pi\biggr)\biggr)\Biggr]\frac{z^{2k}}{(2k)!}\\
&=\te+\te\sum_{k=1}^{\infty}\Biggl[\sum_{j=1}^{2k} (-1)^{k+j}\frac{2^{2k}}{j!}\sum_{\ell=1}^j(-1)^\ell\binom{j}{\ell} \frac{T(2k+\ell,\ell)}{\binom{2k+\ell}{\ell}}\Biggr]\frac{z^{2k}}{(2k)!}.
\end{align*}
Further considering~\eqref{T(j+ell-ell)}, we prove Corollary~\ref{exp-sinc-derv-cor}.
\end{proof}

\section{Series expansions of real powers of sinc and sinhc functions}

In this section, with the aid of Theorem~\ref{part-Bell-sine-thm} and other results in the above sections, we establish series expansions at the point $z=0$ of the power functions $\sinc^rz$ and $\sinhc^rz$ for real $r\in\mathbb{R}$.

\begin{thm}\label{recip-sin-ser-closed-thm}
When $r\ge0$, the series expansions
\begin{equation}\label{recip-sin-ser-closed-eq}
\sinc^rz=1+\sum_{q=1}^{\infty}(-1)^q\Biggl[\sum_{k=1}^{2q}\frac{(-r)_k}{k!}
\sum_{j=1}^k(-1)^j\binom{k}{j} \frac{T(2q+j,j)}{\binom{2q+j}{j}}\Biggr]\frac{(2z)^{2q}}{(2q)!}
\end{equation}
and
\begin{multline}\label{recip-sin-stirl-closed-eq}
\sinc^rz=1+\sum_{q=1}^{\infty}(-1)^q\Biggl[\sum_{k=1}^{2q}\frac{(-r)_k}{k!}
\sum_{j=1}^k(-1)^j\binom{k}{j} \\
\times\sum_{m=0}^{2q}(-1)^{m}\binom{2q}{m} \biggl(\frac{j}{2}\biggr)^{m} \frac{S(2q+j-m,j)} {\binom{2q+j-m}{j}}\Biggr]\frac{(2z)^{2q}}{(2q)!}
\end{multline}
are convergent in $z\in\mathbb{C}$, where the rising factorial $(r)_k$ is defined by
\begin{equation*}
(r)_k=\prod_{\ell=0}^{k-1}(r+\ell)
=
\begin{cases}
r(r+1)\dotsm(r+k-1), & k\ge1;\\
1, & k=0.
\end{cases}
\end{equation*}
When $r<0$, the series expansions~\eqref{recip-sin-ser-closed-eq} and~\eqref{recip-sin-stirl-closed-eq} are convergent in $|z|<\pi$.
\end{thm}

\begin{proof}
By virtue of the Fa\`a di Bruno formula~\eqref{Bruno-Bell-Polynomial}, we obtain
\begin{gather*}
\frac{\td^j(\sinc^rz)}{\td z^j}
=\sum_{k=1}^{j} \frac{\td^ku^r}{\td u^k} B_{j,k}\bigl((\sinc z)', (\sinc z)'', \dotsc, (\sinc z)^{(j-k+1)}\bigr)\\
=\sum_{k=1}^{j} \langle r\rangle_k \sinc^{r-k}z B_{j,k}\bigl((\sinc z)', (\sinc z)'', \dotsc, (\sinc z)^{(j-k+1)}\bigr)\\
\to\sum_{k=1}^{j} \langle r\rangle_k B_{j,k}\biggl(0, -\frac{1}{3},0,\frac{1}{5}, \dotsc, \frac{1}{j-k+2}\sin\frac{(j-k)\pi}{2}\biggr), \quad z\to0\\
=\begin{dcases}
0, & j=2m-1\\
\sum_{k=1}^{2m} \langle r\rangle_k B_{2m,k}\biggl(0, -\frac{1}{3},0,\frac{1}{5}, \dotsc, \frac{1}{j-k+2}\sin\frac{(2m-k)\pi}{2}\biggr), & j=2m\\
\end{dcases}
\end{gather*}
for $m\in\mathbb{N}$, where $u=u(z)=\sinc z$, the notation
\begin{equation*}
\langle r\rangle_k=
\prod_{k=0}^{k-1}(r-k)=
\begin{cases}
r(r-1)\dotsm(r-k+1), & k\ge1\\
1,& k=0
\end{cases}
\end{equation*}
for $r\in\mathbb{R}$ is called the falling factorial, and we used derivatives in~\eqref{sine-deriv-0-Eq}. Therefore, with the help of Theorem~\ref{part-Bell-sine-thm}, we arrive at
\begin{gather*}
\begin{aligned}
\sinc^rz&=1+\sum_{j=1}^{\infty}\biggl[\lim_{z\to0}\frac{\td^j(\sinc^rz)}{\td z^j}\biggr]\frac{z^j}{j!}\\
&=1+\sum_{m=1}^{\infty}\biggl[\lim_{z\to0}\frac{\td^{2m}(\sinc^rz)}{\td z^{2m}}\biggr]\frac{z^{2m}}{(2m)!}
\end{aligned}\\
=1+\sum_{m=1}^{\infty}\Biggl[\sum_{k=1}^{2m} \langle r\rangle_k B_{2m,k}\biggl(0, -\frac{1}{3},0,\frac{1}{5}, \dotsc, \frac{1}{j-k+2}\sin\frac{(2m-k)\pi}{2}\biggr)\Biggr]\frac{z^{2m}}{(2m)!}\\
=1+\sum_{m=1}^{\infty}\Biggl[\sum_{k=1}^{2m}(-1)^{m+k}\frac{\langle r\rangle_k}{k!}\sum_{j=1}^k\binom{k}{j} \frac{1}{2^{j}} \frac{1}{(2m+j)!}\sum_{q=0}^{j}(-1)^q\binom{{j}}{q} (2q-j)^{2m+j}\Biggr]z^{2m}\\
=1+\sum_{m=1}^{\infty}(-1)^m\Biggl[\sum_{k=1}^{2m}\frac{(-r)_k}{k!}
\sum_{j=1}^k(-1)^j\binom{k}{j} \frac{T(2m+j,j)}{\binom{2m+j}{j}}\Biggr]\frac{(2z)^{2m}}{(2m)!}.
\end{gather*}
By virtue of~\eqref{T(j+ell-ell)}, the proof of Theorem~\ref{recip-sin-ser-closed-thm} is thus complete.
\end{proof}

\begin{cor}
For $r\in\mathbb{R}$ and $z\in\mathbb{C}$, we have
\begin{equation}\label{recip-sinh-ser-cl-eq}
\sinhc^rz=1+\sum_{q=1}^{\infty}\Biggl[\sum_{k=1}^{2q}\frac{(-r)_k}{k!}
\sum_{j=1}^k(-1)^j\binom{k}{j} \frac{T(2q+j,j)}{\binom{2q+j}{j}}\Biggr]\frac{(2z)^{2q}}{(2q)!}
\end{equation}
and
\begin{multline}\label{recip-sinh-stirl-cl-eq}
\sinhc^rz=1+\sum_{q=1}^{\infty}\Biggl[\sum_{k=1}^{2q}\frac{(-r)_k}{k!}
\sum_{j=1}^k(-1)^j\binom{k}{j} \\
\times\sum_{m=0}^{2q}(-1)^{m}\binom{2q}{m} \biggl(\frac{j}{2}\biggr)^{m} \frac{S(2q+j-m,j)} {\binom{2q+j-m}{j}}\Biggr]\frac{(2z)^{2q}}{(2q)!}.
\end{multline}
\end{cor}

\begin{proof}
The series expansions~\eqref{recip-sinh-ser-cl-eq} and~\eqref{recip-sinh-stirl-cl-eq} follow from replacing $\sinc z$ by $\sinhc(z\ti)$ in~\eqref{recip-sin-ser-closed-eq} and~\eqref{recip-sin-stirl-closed-eq} and then substituting $z\ti$ for $z$.
\end{proof}

\section{Combinatorial proofs of two identities}
We now modify combinatorial proofs at \url{https://mathoverflow.net/a/420309/} for the identities~\eqref{QGWSID3O} and~\eqref{QGWSID2} as follows.

\begin{proof}[First combinatorial proof of the identity~\eqref{QGWSID3O}]
Theorem~8 in~\cite[p.~247]{Broder} states that weighted Stirling numbers of the second kind
$$
{n+k\brace k}_r=R(n+k-r,n-r,r)
$$
are the monomial symmetric functions of degree $k$ of the integers $r,\dotsc,n$. This means
\begin{equation*}
R(n, k, r)=h_{n-k}(r, r+1, r+2, \dotsc, r+k),
\end{equation*}
where $h_m$ is the complete homogenous symmetric function. So it follows that
$$
R\biggl(2m+j-1,j,-\frac{j}2\biggr)=h_{2m-1}\biggl(-\frac{j}2, 1-\frac{j}2, 2-\frac{j}2, \dotsc, \frac{j}2\biggr)
$$
and we can pair up each monomial $x_{a_1} x_{a_2} x_{a_3} \dotsm$ with
$$
x_{2m-a_1} x_{2m-a_2} x_{2m-a_3} \dotsm=(-1)^{2m-1} x_{a_1} x_{a_2} x_{a_3} \dotsm,
$$
which gives cancellation. The only terms which don't pair up like this with a different term are those which include $x_m=0$ and pair with themselves, but by virtue of including a zero multiplicand they don't contribute anything to the sum. Consequently, by the relation~\eqref{Weight-Stirl-2nd-Cent-Fact-No}, the identity~\eqref{QGWSID3O} is proved.
\end{proof}

\begin{proof}[Second combinatorial proof of the identity~\eqref{QGWSID3O}]
Let
\begin{equation}\label{OperatorT-def}
\mathcal{T}(n,k)= 2^{n-k} T(n,k), \quad n\ge k\ge0
\end{equation}
with $\mathcal{T}(0,0)=1$.
This scaled central triangle number $\mathcal{T}(n,k)$ counts set partitions of $n$ elements into $k$ odd-sized blocks. See the references~\cite{Comtet-French-1972, OEIS-A136630}. This immediately gives the identity~\eqref{QGWSID3O}, since an even-sized set cannot be partitioned into an odd number of odd-sized blocks, nor an odd-sized set partitioned into an even number of odd-sized blocks.
\end{proof}

\begin{proof}[A combinatorial proof of the identity~\eqref{QGWSID2}]
Since
\begin{align*}
\sum_{j=1}^k (-1)^j \binom{k}{j} \frac{T(2\ell+j,j)}{\binom{2\ell+j}{j}}
&=\sum_{j=1}^k (-1)^j \frac{k!(2\ell)!}{(k-j)! (2\ell+j)!} T(2\ell+j,j) \\
&=\frac{(-1)^k}{2^{2\ell} \binom{2\ell+k}{k}} \sum_{j=1}^k (-1)^{k-j} \binom{2\ell+k}{2\ell+j} \mathcal{T}(2\ell+j,j),
\end{align*}
the identity~\eqref{QGWSID2} is equivalent to
\begin{equation}\label{QGWSID2-Equiv}
\sum_{j=1}^k (-1)^{k-j} \binom{2\ell+k}{2\ell+j} \mathcal{T}(2\ell+j,j)=0, \quad 1\le \ell<k,
\end{equation}
where $\mathcal{T}(2\ell+j,j)$ is defined by~\eqref{OperatorT-def}.
The equality~\eqref{QGWSID2-Equiv} has the following combinatorial proof.
\par
Consider set partitions of $2\ell+k$ elements into $k$ odd-sized blocks where blocks of size $3$ or greater are coloured red and singleton blocks can be coloured red or blue. Then the sum counts such set partitions weighted by $(-1)^\text{number of blue partitions}$. Note that $j$ is the number of red partitions. Observe that partitions containing at least one singleton can be paired with the partition which differs only in the colour assigned to the singleton with the smallest element, so that the sum counts the number of partitions of $2\ell+k$ elements into $k$ odd-sized blocks of at least $3$ elements each. But, if $k>\ell$, there are no such partitions. The required proof is complete.
\end{proof}

\section{Remarks}
Finally we list several remarks about our main results and related things.

\begin{rem}
The formulation of the series expansions~\eqref{sine-power-ser-expan-eq} and~\eqref{sinc-power-series-2nd-Eq} in Theorems~\ref{Since-series-expan-lem} and~\ref{sinc-power-series-2nd-thm} are better and simpler than corresponding ones in~\cite[pp.~798--799]{Brychkov-ITSF-2009}.
\par
The formula~\eqref{sin-poer-exp} can also be found at \url{https://math.stackexchange.com/a/4331451/} and \url{https://math.stackexchange.com/a/4332549/}.
\end{rem}

\begin{rem}\label{Gelinas-Rem-Blissard}
After reading the preprint~\cite{v1ser-sine-real-pow.tex} of this paper, Jacques G\'elinas, a retired mathematician at Ottawa in Canada, pointed out that the series expansion~\eqref{sine-power-ser-expan-eq} in Theorem~\ref{Since-series-expan-lem}, or say, the series expansion~\eqref{sinc-power-series-2nd-Eq} in Theorem~\ref{sinc-power-series-2nd-thm}, has been considered by John Blissard in~\cite[pp.~50--51]{Blissard-1864} with different and old notations.
\end{rem}

\begin{rem}
The series expansion~\eqref{sine-power-ser-expan-eq} in Theorem~\ref{Since-series-expan-lem} has been applied to answer questions at the sites \url{https://math.stackexchange.com/a/4429078/}, \url{https://math.stackexchange.com/a/4332549/}, and \url{https://math.stackexchange.com/a/4331451/}.
\par
The series expansion~\eqref{sine-power-ser-expan-eq} in Theorem~\ref{Since-series-expan-lem} or the series expansion~\eqref{recip-sin-ser-closed-eq} in Theorem~\ref{recip-sin-ser-closed-thm} can be used to answer questions at \url{https://math.stackexchange.com/q/2267836/} and \url{https://math.stackexchange.com/q/3673133/}.
\par
The series expansion~\eqref{recip-sin-ser-closed-eq} in Theorem~\ref{recip-sin-ser-closed-thm} has been employed to answer questions at the websites \url{https://math.stackexchange.com/a/4427504/}, \url{https://math.stackexchange.com/a/4426821/}, and~\url{https://math.stackexchange.com/a/4428010/}.
\par
The series expansion~\eqref{recip-sin-ser-closed-eq} in Theorem~\ref{recip-sin-ser-closed-thm} has been utilized in~\cite[Theorem~3]{v2Bernoulli-ID-Stack.tex} to derive two closed-form formulas for the Bernoulli numbers $B_{2m}$ in terms of central factorial numbers of the second kind $T(2m+j,j)$.
\end{rem}

\begin{rem}
The first identity in Theorem~\ref{part-Bell-sine-thm} is a special case of the following general conclusion in~\cite[Theorem~1.1]{AIMS-Math20210491.tex}.
\begin{quote}
For $k,n\in\mathbb{N}_0$, $m\in\mathbb{N}$, and $x_{m}\in\mathbb{C}$, we have
\begin{equation*}
B_{2n+1,k}\biggl(0,x_2,0,x_4,\dotsc,\frac{1+(-1)^{k}}{2}x_{2n-k+2}\biggr)=0.
\end{equation*}
\end{quote}
\end{rem}

\begin{rem}
As done in Corollary~\ref{exp-sinc-derv-cor}, as long as the function $f(u)$ is infinitely differentiable at the point $u=1$, Theorem~\ref{part-Bell-sine-thm} can be utilized to compute series expansions at $x=0$ of the functions $f(\sinc x)$ and $f(\sinhc x)$.
\end{rem}

\begin{rem}
Let $r>0$ and $k\in\mathbb{N}_0$. Making use of the Fa\`a di Bruno formula~\eqref{Bruno-Bell-Polynomial} and employing the formula
\begin{equation*}
B_{n,k}(x,1,0,\dotsc,0)
=\frac{1}{2^{n-k}}\frac{n!}{k!}\binom{k}{n-k}x^{2k-n}
\end{equation*}
collected in~\cite[Section~1.4]{Bell-value-elem-funct.tex}, we obtain
\begin{align*}
\biggl[\frac{1}{(1+x^2)^r}\biggr]^{(k)}&=\sum_{j=0}^{k}\frac{\td^j}{\td u^j}\biggl(\frac{1}{u^r}\biggr) B_{k,j}(2x,2,0,\dotsc,0)\\
&=\sum_{j=0}^{k}\frac{\langle-r\rangle_j}{u^{r+j}} 2^jB_{k,j}(x,1,0,\dotsc,0)\\
&=\sum_{j=0}^{k}\frac{\langle-r\rangle_j}{(1+x^2)^{r+j}} 2^j \frac{1}{2^{k-j}}\frac{k!}{j!}\binom{j}{k-j}x^{2j-k}\\
&=\frac{k!}{2^kx^k(1+x^2)^r} \sum_{j=0}^{k}\langle-r\rangle_j\frac{2^{2j}}{j!}\binom{j}{k-j}\frac{x^{2j}}{(1+x^2)^{j}},
\end{align*}
where $u=u(x)=1+x^2$. See also texts at the site \url{https://math.stackexchange.com/a/4418636/}.
\end{rem}

\begin{rem}
We would like to mention the papers~\cite{Sinc-Bound-Chen-RACSAM-2022, Qian-Sinc-Ineq-RACSAM-2022, Zhu-Racsam-81-2020}, in which the power function $\sinc^rz$ for some specific ranges of $r,x\in\mathbb{R}$ is bounded from both sides, and to mention the papers~\cite{jordan-generalized-simp.tex, Gene-Jordan-Inequal.tex, refine-jordan-kober.tex}, in which many bounds of the sinc function $\sinc x$ for $x\in\bigl(0,\frac{\pi}{2}\bigr)$ are established, reviewed, and surveyed.
\end{rem}

\begin{rem}
This paper is a revised version of the electronic arXiv preprint~\cite{ser-sinc-real-pow.tex}.
\end{rem}

\section{Declarations}

\subsection{Acknowledgements}
The authors thank Dr. Jacques G\'elinas, a retired mathematician at Ottawa in Canada, for his hard efforts to look up closely-related references and for his valuable discussions, including those mentioned in Remark~\ref{Gelinas-Rem-Blissard} of this paper.

\subsection{Availability of data and material}
Data sharing is not applicable to this article as no new data were created or analyzed in this study.

\subsection{Competing interests}
The authors declare that they have no any conflict of competing interests.

\subsection{Authors' contributions}
All authors contributed equally to the manuscript and read and approved the final manuscript.

\subsection{Funding}
Not applicable.

\end{document}